\documentclass[reqno]{amsart}

\usepackage[utf8]{inputenc}
\usepackage{amsthm}
\usepackage{amsmath}
\usepackage{amssymb}
\usepackage{enumitem}
\usepackage[dvips]{graphicx}
\usepackage{color}
\usepackage{hyperref}
\usepackage{url}
\usepackage[left=3.5cm, right=3.5cm, paperheight=11.4in]{geometry}
\usepackage{fancyhdr}
\usepackage{bm}
\usepackage{mathrsfs}
\usepackage{comment}
\usepackage{stmaryrd}
\usepackage{nicefrac}
\usepackage{soul}

\newtheorem{theorem}{Theorem}[section]
\newtheorem{lemma}[theorem]{Lemma}
\newtheorem{corollary}[theorem]{Corollary}
\newtheorem*{theorem*}{Main Theorem}

\theoremstyle{definition}

\newtheorem{remark}[theorem]{\textsc{Remark}}

\hypersetup{
    pdftitle={Subcommutativity of integrals and quasi-arithmetic means},
    pdfauthor={G\l{}azowska, Leonetti, Matkowski, and Tringali},
    pdfmenubar=false,
    pdffitwindow=true,
    pdfstartview=FitH,
    colorlinks=true,
    linkcolor=blue,
    citecolor=green,
    urlcolor=black
}

\renewcommand{\emptyset}{\varnothing}

\newcommand\starop[1]{\star}

\newcommand{\fixed}[2][1]{%
  \begingroup
  \spaceskip=#1\fontdimen2\font minus \fontdimen4\font
  \xspaceskip=0pt\relax
  #2%
  \endgroup
}

\begin{document}
\title{Subcommutativity of integrals and quasi-arithmetic means}

\author{Dorota G\l azowska}
\address[Dorota~G\l azowska]{Institute of Mathematics, University of Zielona G\'ora -- prof. Z. Szafrana 4a, PL-65-516 Zielona G\'ora}
\email{d.glazowska@im.uz.zgora.pl}

\author{Paolo Leonetti}
\address[Paolo~Leonetti]{Department of Economics, Universit\`a degli Studi dell'Insubria -- via Monte Generoso 71, IT-21100 Varese}
\email{leonetti.paolo@gmail.com}
\urladdr{\url{https://sites.google.com/site/leonettipaolo/}}

\author{Janusz Matkowski}
\address[Janusz~Matkowski]{Institute of Mathematics, University of Zielona G\'ora -- prof. Z. Szafrana 4a, PL-65-516 Zielona G\'ora}
\email{j.matkowski@im.uz.zgora.pl}
\urladdr{\url{http://januszmatkowski.com}}

\author{Salvatore Tringali}
\address[Salvatore~Tringali]{School of Mathematical Sciences,
Hebei Normal University | Shijiazhuang, Hebei province, 050024 China}
\email{salvo.tringali@gmail.com}
\urladdr{http://imsc.uni-graz.at/tringali}

%
%


\begin{abstract}
Let $(X, \mathscr{L}, \lambda)$ and $(Y, \mathscr{M}, \mu)$ be finite measure spaces for which there exist $A \in \mathscr{L}$ and $B \in \mathscr{M}$ with either $0 < \lambda(A) < 1 < \lambda(X)$ and $0 < \mu(B) < \mu(Y)$, or the other way around. In addition, let $I \subseteq \mathbb{R}$ be a non-empty open interval, and suppose that $f,g\colon I \to \mathbb{R}_{+}$ are
homeo\-morphisms with $g$ increasing. We prove that the functional inequality
$$
f^{-1}\!\left(\int_X f\!\left(g^{-1}\!\left(\int_Y g \circ h\;d\mu\right)\right)d \lambda\right)\! \le g^{-1}\!\left(\int_Y g\!\left(f^{-1}\!\left(\int_X f \circ h\;d\lambda\right)\right)d \mu\right)
$$
is satisfied by every $\mathscr{L} \otimes \mathscr{M}$-measurable simple function $h: X \times Y \to I$ if and only if $f=a \fixed[0.15]{\text{ }} g^b$ for some $a,b \in \mathbb{R}_{+}$ with $b\ge 1$. An analogous characterization is given for probability spaces.  
\end{abstract}
\maketitle
\thispagestyle{empty}
%

\section{Introduction and Main Results}
\label{sec:intro}
Let $(X,\mathscr{L},\lambda)$ and $(Y,\mathscr{M},\mu)$ be finite measure spaces, $I \subseteq \mathbb R$ be a non-empty interval, and $f, g\colon I\to \mathbb{R}$ be continuous in\-jec\-tions (note that $I$ may be bounded or unbounded, and need be neither open nor closed). We look for conditions on $f$ and $g$ under which the inequality
\begin{equation}\label{eq:maininequality}
f^{-1}\!\left(\int_X f\!\left(g^{-1}\!\left(\int_Y g \circ h\;d\mu\right)\right)d \lambda\right)\! \le g^{-1}\!\left(\int_Y g\!\left(f^{-1}\!\left(\int_X f \circ h\;d\lambda\right)\right)d \mu\right)\!
\end{equation}
is satisfied by every $h\colon X\times Y\rightarrow I$ in a suitable class of $\mathscr{L} \otimes \mathscr{M}$-measurable functions, taking for granted that 
each side of \eqref{eq:maininequality} is well defined. 

If $(X,\mathscr{L},\lambda)$ and $(Y,\mathscr{M},\mu)$ are probability spaces, the left- and right-hand side of \eqref{eq:maininequality} can be interpreted as ``partially mixed integral quasi-arithmetic means.'' 
Quasi-arithmetic means, commonly known as ``certainty equivalents'' in decision theory, were first considered by Kolmogorov \cite{Kol30}, Nagumo \cite{Nag30}, de Finetti \cite{deFin31}, and Kitagawa \cite{Kit34}, and have been proved to be useful in a large variety of contexts, see e.g. \cite{MR2887277, MR2171327, Leo22, MR2154023, MR4102319, MR4244243}. 
We refer the reader to \cite[Sect.~1]{LMT16} for related literature on the problem. 

It has been shown in \cite[Proposition 2]{LMT16} that each side of \eqref{eq:maininequality} is well posed if $(X,\mathscr{L},\lambda)$ and $(Y,\mathscr{M},\mu)$ are probability spaces and the image $h(X\times Y)$ is contained in a compact subset of $I$ for every ``test function'' $h$, which is especially the case when $h\colon X \times Y \to I$ is an $\mathscr{L} \otimes \mathscr{M}$-measurable simple function.

A related question has been recently addressed in \cite{GLMT17}, considering the equality case in non-degenerate probability spaces
\begin{equation}\label{eq:mainequation}
f^{-1}\!\left(\int_X f\!\left(g^{-1}\!\left(\int_Y g \circ h\;d\mu\right)\right)d \lambda\right)\! = g^{-1}\!\left(\int_Y g\!\left(f^{-1}\!\left(\int_X f \circ h\;d\lambda\right)\right)d \mu\right)\!.
\end{equation}
(All measure spaces are assumed to be finite, and a measure space $(S,\mathscr{C},\gamma)$ is said to be \emph{non-degenerate} if there exists $A \in \mathscr{C}$ with $0<\gamma(A)<\gamma(S)$.)

By the main result in \cite{LMT16}, the following characterization holds:
\begin{theorem}\label{th:mainoldold}
Let $(X,\mathscr{L},\lambda)$ and $(Y,\mathscr{M},\mu)$ be non-degenerate probability spaces, and $f,g: I\to \mathbb{R}$ be continuous injections. 

Then equation \eqref{eq:mainequation} is satisfied by every $\mathscr{L}\otimes \mathscr{M}$-measurable simple func\-tion $h:X\times Y \to I$ if and only if $f=ag+b$ for some $a,b \in \mathbb{R}$ with $a\neq 0$.
\end{theorem}

Moreover, if $(X,\mathscr{L},\lambda)$ and $(Y,\mathscr{M},\mu)$ are not probability spaces, then it is easy to check that both sides of \eqref{eq:mainequation} are well defined if $f$ and $g$ are continuous bijections from an interval $I$ to the positive reals $\mathbb{R}_{+}:=(0,\infty)$. 
In such a case, the interval $I$ is necessarily open so that, thanks to the Brouwer's invariance of domain theorem, see \cite{Bro1}, both $f$ and $g$ are homeomorphisms. 

Accordingly, the following result has been proved by the authors in \cite[Theorem 2]{GLMT17}: 
\begin{theorem}\label{th:mainold}
Let $(X,\mathscr{L},\lambda)$ and $(Y,\mathscr{M},\mu)$ be non-degenerate measure spaces, and  $f,g: I\to \mathbb{R}_{+}$ be continuous bijections. 

Then equation \eqref{eq:mainequation} is satisfied by every $\mathscr{L}\otimes \mathscr{M}$-measurable simple function $h:X\times Y \to I$ if and only if $f=c \fixed[0.15]{\text{ }}g$ for some $c\in \mathbb{R}_{+}$.
\end{theorem}

Now we may ask what happens considering the (more general) functional inequality \eqref{eq:maininequality}, and we provide an answer to this question in several cases. 
Setting 
$$
\mathcal{M}_f(h):=f^{-1}\left(\int_X f \circ h \, \mathrm{d}\lambda \right)
\quad \text{ and }\quad
\mathcal{M}_g(h):=g^{-1}\left(\int_Y g \circ h \, \mathrm{d}\mu \right)
$$
for each $\mathscr{L}\otimes \mathscr{M}$-measurable simple function $h: X\times Y\to I$, equality \eqref{eq:mainequation} can be rewritten as $\mathcal{M}_f \circ \mathcal{M}_g = \mathcal{M}_g \circ \mathcal{M}_f$, i.e., the operators $\mathcal{M}_f$ and $\mathcal{M}_g$ commute. 
Analogously, inequality \eqref{eq:maininequality} holds if $\mathcal{M}_f \circ \mathcal{M}_g \le  \mathcal{M}_g \circ \mathcal{M}_f$, which can be interpreted as the \textquotedblleft subcommutativity of the pair 
$(\mathcal{M}_f, \mathcal{M}_g)$\textquotedblright\, 
or 
\textquotedblleft supercommutativity of the pair 
$(\mathcal{M}_g, \mathcal{M}_f)$.\textquotedblright\, 
A somehow related subcommutativity functional inequality has been studied in \cite{MR3474976}. 
%
%
\subsection{Main Results}\label{sec:mainresult}

Our first main result follows, which applies to a special kind of non-degenerate measure spaces.

\begin{theorem}\label{th:main}
Let $(X,\mathscr{L},\lambda)$ and $(Y,\mathscr{M},\mu)$ be non-degenerate measure spaces for which there exists $A \in \mathscr{L}$ such that $0 < \lambda(A) < 1 < \lambda(X)$, and let $f,g \colon I \to \mathbb{R}_{+}$ be 
continuous bijections. 

Then the inequality \eqref{eq:maininequality} is satisfied by every $\mathscr{L}\otimes \mathscr{M}$-measurable simple func\-tion $h \colon X\times Y \to I$ if and only if one the following holds\text{:}
\begin{enumerate}[label={\rm (\textsc{l}\arabic{*})}]
\item \label{item:1thm1} $g$ is increasing and $f=a \fixed[0.15]{\text{ }} g^b$ for some $a,b \in \mathbb{R}_{+}$ with $b\ge 1$\textup{;}
\item \label{item:2thm1} $g$ is decreasing and $f=a \fixed[0.15]{\text{ }} g^b$ for some $a \in \mathbb{R}_{+}$ and some non-zero $b\le 1$\textup{.}
\end{enumerate}
\end{theorem}

An analogue characterization holds replacing the hypothesis $0<\lambda(A)<1<\lambda(X)$ for some measurable set $A$ from the first to the second space (we state it for future references):
\begin{theorem}\label{th:main2}
Let $(X,\mathscr{L},\lambda)$ and $(Y,\mathscr{M},\mu)$ be non-degenerate measure spaces for which there exists $B \in \mathscr{M}$ such that $0 < \lambda(B) < 1 < \lambda(Y)$, 
and let $f,g: I\to \mathbb{R}_{+}$ be 
continuous bijections. 

Then the inequality \eqref{eq:maininequality} is satisfied by every $\mathscr{L}\otimes \mathscr{M}$-measurable simple func\-tion $h:X\times Y \to I$ if and only if 
\ref{item:1thm1} or \ref{item:2thm1} holds\textup{.} 
\end{theorem}

The proofs of the Theorem \ref{th:main} and Theorem \ref{th:main2} go essentially on the same lines. 
In addition, we are going to show in Lemma \ref{lem:keylem} below that, if inequality \eqref{eq:maininequality} holds and $g$ is increasing on arbitrary non-degenerate measure spaces, then $f$ is increasing as well. On the other hand, if $g$ is decreasing, then $f$ is not necessarily decreasing by Remark \ref{rmk:rmkcounterexample} (in a simple case which is not covered by our two results). 

In the case where $g$ is increasing, we can merge the above two characterization in the following immediate consequence (which is stated in the abstract):
\begin{corollary}
Let $(X, \mathscr{L}, \lambda)$ and $(Y, \mathscr{M}, \mu)$ be non-degenerate measure spaces for which there exist $A \in \mathscr{L}$ with $0 < \lambda(A) < 1 < \lambda(X)$ or  $B \in \mathscr{M}$ with $0 < \mu(B) < 1 < \mu(Y)$. Let also $f,g: I\to \mathbb{R}_{+}$ be continuous bijections with $g$ increasing. 

Then the inequality \eqref{eq:maininequality} is satisfied by every $\mathscr{L}\otimes \mathscr{M}$-measurable simple func\-tion $h:X\times Y \to I$ if and only if $f=a \fixed[0.15]{\text{ }} g^b$ for some $a,b \in \mathbb{R}_{+}$ with $b\ge 1$.
\end{corollary}

We conclude with the analogue characterization for probability spaces.  
\begin{theorem}\label{thm:probabilitycharacterization}
Let $(X, \mathscr{L}, \lambda)$ and $(Y,\mathscr{M},\mu)$ be a non-degenerate probability spaces, and $f,g: I\to \mathbb{R}_{+}$ be continuous bijections with $g$ increasing. Set $\phi:=f\circ g^{-1}$ and, for each $t \in (0,1)$, define the map $\Psi_t: \mathbb{R}^{2}_{+} \to \mathbb{R}_{+}$ by 
$$
\Psi_{t}(x_1,x_2):=\phi^{-1}(t \phi(x_1)+(1-t)\phi(x_2)),
\quad 
\text{for all }
\left(x_1,x_2\right)
\in \mathbb{R}_{+}^2.
$$ 
Then the following are pairwise equivalent\textup{:}
\begin{enumerate}[label={\rm (\textsc{p}\arabic{*})}]
\item \label{item:prob1} the inequality \eqref{eq:maininequality} is satisfied by every $\mathscr{L}\otimes \mathscr{M}$-measurable simple func\-tion $h:X\times Y \to I$\textup{;}
\item \label{item:prob2} $f$ is increasing and $\Psi_{t}$ is convex for some $t\in (0,1)$\textup{;}
\item \label{item:prob3} $f$ is increasing and $\Psi_{t}$ is convex for every $t\in (0,1)$\textup{.}
\end{enumerate}
\end{theorem}


A strictly related result for classical quasi-arithmetic means can be found in \cite[Chapter IV, Section 5, Theorem 1]{MR2024343}. Note that \cite[Chapter IV, Section 5, Theorem 7]{MR2024343} provides a simple condition which ensures the convexity of the functions $\Psi_{t}$. 

The proof of our main results are given in the next section.

\section{Proofs of the Main Results}

Throughout this section, $f$ and $g$ are continuous bijections from a non-empty interval $I \subseteq \mathbb R$ to $\mathbb{R}_{+}$, and we use the following notation 
\begin{equation}\label{eq:definitionD}
D:=\mathbb{R}_{+}\times \mathbb{R}_{+}
\quad \text{ and }\quad
\varphi:=g\circ f^{-1}.
\end{equation}
Vectors in $D$ will be written with bold letters, e.g., $\mathbf{x}, \mathbf{y} \in D$ means $\mathbf{x}=(x_1,x_2)$ and $\mathbf{y}=(y_1,y_2)$ for some $x_1,x_2,y_1,y_2 \in \mathbb{R}_{+}$. 

Before we proceed to the proofs of the main results, we show a common preliminary lemma. 

\begin{lemma}\label{lem:keylem}
Let $(X,\mathscr{L},\lambda)$ and $(Y,\mathscr{M},\mu)$ be measure spaces for which there exist $A \in \mathscr{L}$ and $B \in \mathscr{M}$ with $0 < \lambda(A) < \lambda(X)$ and $0 < \mu(B) < \mu(Y)$, and let $f,g: I\to \mathbb{R}_{+}$ be continuous bijections. 
To ease the notation, define 
\begin{equation}\label{eq:alphai}
\alpha_1:=\lambda(A),\quad 
\alpha_2:=\lambda(A^c),\quad
\beta_1:=\mu(B)
\quad \text{ and }\quad
\beta_2:=\mu(B^c),
\end{equation}
where $A^c:= X\setminus A$ and $B^c:=Y\setminus B$. 

Suppose that 
inequality \eqref{eq:maininequality} is satisfied by every 
$\mathscr{L}\otimes \mathscr{M}$-measurable simple func\-tion $h:X\times Y \to I$. Then the following hold\textup{:}
\begin{enumerate}[label={\rm (\roman{*})}]
\item \label{item:1lem} If $g$ is increasing then $f$ is increasing\textup{;}
\item \label{item:2lem} If $g$ is increasing or if $f$ is increasing and $g$ is decreasing then 
\begin{equation}\label{eq:fundamentalinequality}
\alpha_1 \Phi(\mathbf{x})+\alpha_2 \Phi(\mathbf{y})\le \Phi(\alpha_1 \mathbf{x}+\alpha_2 \mathbf{y}),
\quad \text{ for all }\mathbf{x}, \mathbf{y} \in D,
\end{equation}
where
$\Phi: D\to \mathbb{R}_{+}$ is the function defined by 
$\Phi(\mathbf{x}):=\varphi^{-1}(\beta_1 \varphi(x_1)+\beta_2\varphi(x_2))$\textup{;} if both $f$ and $g$ are decreasing then inequality \eqref{eq:fundamentalinequality} holds with the opposite direction\textup{;} 
\item \label{item:3lem} If $g$ is increasing then 
\begin{equation}\label{eq:fundamentalinequality3}
\beta_1 \tilde{\Phi}(\mathbf{x})+\beta_2 \tilde{\Phi}(\mathbf{y})\le \tilde{\Phi}(\beta_1 \mathbf{x}+\beta_2 \mathbf{y}),
\quad \text{ for all }\mathbf{x}, \mathbf{y} \in D,
\end{equation}
where $\tilde{\Phi}: D\to \mathbb{R}_{+}$ is the function defined by $\tilde{\Phi}(\mathbf{x}):=-\varphi(\alpha_1 \varphi^{-1}(x_1)+\alpha_2\varphi^{-1}(x_2))$\textup{;} if $g$ is decreasing then inequality \eqref{eq:fundamentalinequality3} holds with the opposite direction\textup{.}
\end{enumerate}
\end{lemma}

\begin{proof}
For all $x,y,z,w\in I$ the function
\begin{equation}
\label{equ:expression-of-h-as-sum-of-simple-fncts}
h=x\bm{1}_{A\times B}+y\bm{1}_{A\times B^c}+z\bm{1}_{A^c\times B} + w\bm{1}_{A^c\times B^c}
\end{equation}
is an $\mathscr{L}\otimes \mathscr{M}$-measurable simple function from $X\times Y$ to $I$, so we can put \eqref{equ:expression-of-h-as-sum-of-simple-fncts} into \eqref{eq:maininequality} and obtain
\begin{equation}
\label{equ:main1}
\begin{split}
f^{-1} (\alpha_1 f(g^{-1}(\beta_1&\fixed[0.2]{\text{ }} g(x)+\beta_2\fixed[0.2]{\text{ }} g(y))) + \alpha_2 f(g^{-1}(\beta_1 \fixed[0.2]{\text{ }}g(z)+\beta_2\fixed[0.2]{\text{ }}g(w)))) \\
& \le g^{-1}(\beta_1\fixed[0.2]{\text{ }} g(f^{-1}(\alpha_1 f(x)+\alpha_2f(z))) + \beta_2\fixed[0.2]{\text{ }} g(f^{-1}(\alpha_1 f(y)+\alpha_2f(w)))).
\end{split}
\end{equation}

First, suppose that $g$ is increasing. 
Note that $\varphi=g\circ f^{-1}$ is a continuous bijection on $\mathbb{R}_{+}$. Now, applying $g$ on both sides of \eqref{equ:main1} and setting 
$$
x= f^{-1}(s), \quad y=f^{-1}(t), \quad z= f^{-1}(u),
\quad  \text{and} \quad 
w = f^{-1}(v),
$$
we get
\begin{equation}\label{eq:increasingf}
\begin{split}
\varphi( \alpha_1 \fixed[0.2]{\text{ }} \varphi^{-1}(\beta_1 \fixed[0.2]{\text{ }}\varphi(s) + \beta_2\fixed[0.2]{\text{ }}\varphi(t)) &+ \alpha_2\fixed[0.2]{\text{ }} \varphi^{-1}(\beta_1 \fixed[0.2]{\text{ }}\varphi(u) + \beta_2\fixed[0.2]{\text{ }} \varphi(v)))  \\
             & \le 
             \beta_1\fixed[0.2]{\text{ }}\varphi(\alpha_1\fixed[0.2]{\text{ }} s+\alpha_2\fixed[0.2]{\text{ }}u) + \beta_2\fixed[0.2]{\text{ }}\varphi(\alpha_1\fixed[0.2]{\text{ }} t+\alpha_2\fixed[0.2]{\text{ }}v)
\end{split}
\end{equation}
for every $s,t,u,v \in f(I)=\mathbb{R}_{+}$. 

Let us suppose for the sake of contradiction that $f$ is decreasing. Then $\varphi$ would be a decreasing homeomorphism on $\mathbb{R}_{+}$, so that $\varphi(r)\to +\infty$ as $r\to 0^+$. Setting $s=v=1$ and applying $\varphi^{-1}$ on both sides of \eqref{eq:increasingf}, we obtain
\begin{displaymath}
\begin{split}
\alpha_1 \fixed[0.2]{\text{ }} \varphi^{-1}(\beta_1 \fixed[0.2]{\text{ }}\varphi(1) + \beta_2\fixed[0.2]{\text{ }}\varphi(t)) &+ \alpha_2\fixed[0.2]{\text{ }} \varphi^{-1}(\beta_1 \fixed[0.2]{\text{ }}\varphi(u) + \beta_2\fixed[0.2]{\text{ }} \varphi(1))  \\
             & \ge 
             \varphi^{-1} (
             \beta_1\fixed[0.2]{\text{ }}\varphi(\alpha_1\fixed[0.2]{\text{ }} +\alpha_2\fixed[0.2]{\text{ }}u) + \beta_2\fixed[0.2]{\text{ }}\varphi(\alpha_1\fixed[0.2]{\text{ }} t+\alpha_2\fixed[0.2]{\text{ }})
             )
\end{split}
\end{displaymath}
for every $t,u \in \mathbb{R}_{+}$. Letting $t\to 0^+$, $u\to 0^+$ and using the continuity of $\varphi$, we conclude that 
$$\varphi^{-1} (
             \beta_1\fixed[0.2]{\text{ }}\varphi(\alpha_1) + \beta_2\fixed[0.2]{\text{ }}\varphi(\alpha_2\fixed[0.2]{\text{ }})
             )\le 0,
$$
which is a contradiction. 
%
Therefore both $f$ and $\varphi$ are increasing, which proves \ref{item:1lem}. 

Applying $\varphi^{-1}$ on both sides of \eqref{eq:increasingf} we get
\begin{equation*}
\begin{split}
\alpha_1 \fixed[0.2]{\text{ }} \varphi^{-1}(\beta_1 \fixed[0.2]{\text{ }}\varphi(s) + \beta_2\fixed[0.2]{\text{ }}\varphi(t)) &+ \alpha_2\fixed[0.2]{\text{ }} \varphi^{-1}(\beta_1 \fixed[0.2]{\text{ }}\varphi(u) + \beta_2\fixed[0.2]{\text{ }} \varphi(v))  \\
             & \le 
             \varphi^{-1} \left(
             \beta_1\fixed[0.2]{\text{ }}\varphi(\alpha_1\fixed[0.2]{\text{ }} s+\alpha_2\fixed[0.2]{\text{ }}u) + \beta_2\fixed[0.2]{\text{ }}\varphi(\alpha_1\fixed[0.2]{\text{ }} t+\alpha_2\fixed[0.2]{\text{ }}v) \right)
\end{split}
\end{equation*}
for every $t,u \in \mathbb{R}_{+}$, and recalling the definition of 
$\Phi$, we obtain the inequality \eqref{eq:fundamentalinequality}, which proves also the first part of item \ref{item:2lem}. 
%
Similarly, taking the value of $g$ for both sides of \eqref{equ:main1}, and setting 
$$
x= g^{-1}(s), \quad y=g^{-1}(t), \quad z= g^{-1}(u),\quad  \text{and} 
 \quad w = g^{-1}(v),
$$ 
we obtain 
\begin{equation}\label{eq:increasingf2}
\begin{split}
\varphi( \alpha_1 \fixed[0.2]{\text{ }} \varphi^{-1}(\beta_1 \fixed[0.2]{\text{ }}s + \beta_2\fixed[0.2]{\text{ }}t) &+ \alpha_2\fixed[0.2]{\text{ }} \varphi^{-1}(\beta_1 \fixed[0.2]{\text{ }}u + \beta_2\fixed[0.2]{\text{ }} v))  \\
             & \le 
             \beta_1\fixed[0.2]{\text{ }}\varphi(\alpha_1\fixed[0.2]{\text{ }} \varphi^{-1}(s)+\alpha_2\fixed[0.2]{\text{ }}\varphi^{-1}(u)) + \beta_2\fixed[0.2]{\text{ }}\varphi(\alpha_1\fixed[0.2]{\text{ }} \varphi^{-1}(t)+\alpha_2\fixed[0.2]{\text{ }}\varphi^{-1}(v))
\end{split}
\end{equation}
for every $s,t,u,v \in g(I)=\mathbb{R}_{+}$. 
By the definition of the function $\tilde{\Phi}$, this is equivalent to \eqref{eq:fundamentalinequality3}, which proves the first part of item \ref{item:3lem}. 

Now, assume that $g$ is decreasing and note that, in such case, inequality \eqref{eq:increasingf} needs to be reversed. If $f$ is increasing then both $\varphi$ and $\varphi^{-1}$ are decreasing and, consequently, we obtain inequality \eqref{eq:fundamentalinequality}. Viceversa, if $f$ is also decreasing, both $\varphi$ and $\varphi^{-1}$ are increasing, so the inequality \eqref{eq:fundamentalinequality} holds with the opposite direction. This completes the proof of item \ref{item:2lem}.
Similarly, since $g$ is decreasing, inequality \eqref{eq:increasingf2} needs to be reversed and, hence, also inequality \eqref{eq:fundamentalinequality3}, which completes the proof of item \ref{item:3lem}. 
\end{proof}

\begin{remark}\label{rmk:rmkcounterexample}
One may be tempted to guess from the statement of Lemma \ref{lem:keylem}\ref{item:1lem} that if $g$ is decreasing then also $f$ is decreasing. However, this is false. For, suppose that $I=\mathbb{R}_{+}$, $f(t)=t$ and $g(t)=1/t$ for all $t \in\mathbb{R}_{+}$, $\mathscr{L}=\{\emptyset,A,A^c,X\}$ and $\mathscr{M}=\{\emptyset, B,B^c,Y\}$ with $\alpha_1=\alpha_2=\beta_1=\beta_2=1$. In such case, since $g$ is decreasing, then inequality \eqref{eq:maininequality} is satisfied by every $\mathscr{L}\otimes \mathscr{M}$-measurable simple function $h:X\times Y\to I$ if and only if \eqref{eq:fundamentalinequality3} holds with the opposite direction. And the latter one is satisfied, indeed $g=\varphi=\varphi^{-1}$, therefore 
\begin{displaymath}
\begin{split}
\tilde{\Phi}(\mathbf{x})+\tilde{\Phi}(\mathbf{y})-\tilde{\Phi}(\mathbf{x}+\mathbf{y})
&=-\left(\frac{1}{x_1}+\frac{1}{x_2}\right)^{-1}-\left(\frac{1}{y_1}+\frac{1}{y_2}\right)^{-1}+\left(\frac{1}{x_1+y_1}+\frac{1}{x_2+y_2}\right)^{-1}\\
&=\frac{(x_1y_2-x_2y_1)^2}{(x_1+x_2+y_1+y_2)(x_1+x_2)(y_1+y_2)}
\ge 0
\end{split}
\end{displaymath}
for all $\mathbf{x}, \mathbf{y} \in D$.
\end{remark}

In the proof of Theorem \ref{th:main}, we will also need the following characterization:
\begin{lemma}\label{lem:matk}
Fix $\gamma_1,\gamma_2 \in \mathbb{R}_{+}$ such that $0<\gamma_1<1<\gamma_1+\gamma_2$. Let $\Psi \colon D\to \mathbb{R}$ be a function such that $\liminf_{\bm{x}\to (0,0)}\Psi(\mathbf{x})\ge 0$. Then 
\begin{equation}\label{equ:characterization-31dec}
\gamma_1 \Psi(\mathbf{x})+\gamma_2 \Psi(\mathbf{y})\le \Psi(\gamma_1 \mathbf{x}+\gamma_2 \mathbf{y}),
\,\text{ for all }\mathbf{x}, \mathbf{y} \in D,
\end{equation}
if and only if both the following conditions hold\textup{:}
\begin{enumerate}[label={\rm (\textsc{m}\arabic{*})}, resume]
\item \label{item:M1} $\Psi$ is positively homogeneous, i.e., 
$$
\Psi(\gamma \mathbf{x})=\gamma \Psi(\mathbf{x}),\, \text{ for all }\mathbf{x} \in D \text{ and }\gamma \in \mathbb{R}_{+}\textup{;}
$$
\item \label{item:M2} $\Psi$ is superadditive, i.e., 
$$
\Psi(\mathbf{x}+\mathbf{y})\ge \Psi(\mathbf{x})+\Psi(\mathbf{y}),\, \text{ for all }\mathbf{x}, \mathbf{y} \in D\textup{.}
$$
\end{enumerate}
In particular, if $\Psi$ satisfies condition \eqref{equ:characterization-31dec}, then the section $\Psi(\cdot\,,1)$ is concave. 
\end{lemma}
\begin{proof}
The equivalence between \eqref{equ:characterization-31dec} and conditions \ref{item:M1} and \ref{item:M2} follows from
\cite[Corollary 2]{MR1316490} (see also \cite[Theorem 2]{m444}). As for the remaining part of the statement, 
let $h\colon \mathbb{R}_{+}\to \mathbb{R}$ be the function defined by
$$
h\left(x\right):=\Psi\left(x,1\right), 
\text{ for all } 
x \in \mathbb{R}_+. 
$$ 
Then condition \ref{item:M1} implies that $\Psi(\mathbf{x})=x_2 \fixed[0.1]{\text{ }} h(x_1/x_2)$ for all $\mathbf{x}=\left(x_1, x_2\right)\in D$, so that condition \ref{item:M2} is equivalent to
\begin{equation}\label{eq:concavityh}
(x_2+y_2)h\!\left(\frac{x_1+y_1}{x_2+y_2}\right)\ge 
x_2h\!\left(\frac{x_1}{x_2}\right)+y_2h\!\left(\frac{y_1}{y_2}\right),
\quad \text{for all}
\quad 
\mathbf{x}, \mathbf{y} 
\in D\textup{.}
\end{equation} 
Now, fix $\alpha \in (0,1)$ and $x,y \in \mathbb{R}_{+}$. Setting $\mathbf{x}=(\alpha x, \alpha)$ and $\mathbf{y}=((1-\alpha) y, 1-\alpha)$ in \eqref{eq:concavityh}, we conclude that $h$ is concave.  
\end{proof}

Here and after, we continue with the same notation given in \eqref{eq:definitionD} and in the statement of Lemma \ref{lem:keylem} in \eqref{eq:alphai}. 
We are ready for the proof of Theorem \ref{th:main}. 
\begin{proof}[Proof of Theorem \ref{th:main}]
Suppose for now that $g$ is increasing. 

\medskip 

\textsc{Only If part.} By hypothesis, there exists $A \in \mathscr{L}$ such that $\alpha_1:=\lambda(A)$ and $\alpha_2:=\lambda(A^c)$ with $0<\alpha_1<1<\alpha_1+\alpha_2$.  




Thanks to Lemma \ref{lem:keylem}\ref{item:2lem}, 
inequality \eqref{eq:fundamentalinequality} holds. 
It follows by Lemma \ref{lem:matk} that $\Phi$, defined in Lemma \ref{lem:keylem}\ref{item:2lem}, is positively homogeneous, 
which can be rewritten more explicitly as
\begin{equation}\label{eq:homogeneous}
\varphi^{-1}(\beta_1 \fixed[0.2]{\text{ }}\varphi(\gamma x) + \beta_2\fixed[0.2]{\text{ }}\varphi(\gamma y))=\gamma \varphi^{-1}(\beta_1 \fixed[0.2]{\text{ }}\varphi(x) + \beta_2\fixed[0.2]{\text{ }}\varphi(y)),\text{ for all }x,y,\gamma \in \mathbb{R}_{+}.
\end{equation}

Now fix $\gamma \in \mathbb{R}_{+}$ and define the continuous function $\Lambda_\gamma: \mathbb{R}_{+} \to \mathbb{R}_{+}$ by
\begin{equation}\label{eq:defbiglambda}
\Lambda_\gamma(x):=
\varphi(\gamma \fixed[0.2]{\text{ }}\varphi^{-1}(x)), 
\quad \text{ for all }x \in \mathbb{R}_{+}. 
\end{equation} 
Then, applying $\varphi$ to both sides of \eqref{eq:homogeneous} and setting $x= \varphi^{-1}(s)$ and $y= \varphi^{-1}(t)$, we obtain
\begin{equation}\label{eq:lambda}
\beta_1 \fixed[0.2]{\text{ }}\Lambda_\gamma(s) + \beta_2\fixed[0.2]{\text{ }}\Lambda_\gamma(t)=\Lambda_\gamma(\beta_1 \fixed[0.2]{\text{ }}s + \beta_2\fixed[0.2]{\text{ }}t),
\end{equation}
for all $s,t \in \varphi(\mathbb{R}_{+})=\mathbb{R}_{+}$.




It is easy to show that $\Lambda_\gamma(x) \to 0$ as $x \to 0^{+}$. Indeed, note that $\varphi$ is increasing if and only if so is $\varphi^{-1}$. Hence, if $\varphi$ is increasing [respectively, decreasing], then $x\to 0^{+}$ if and only if $\gamma \varphi^{-1}(x)\to 0^{+}$ [resp., $\gamma \varphi^{-1}(x)\to \infty$]. Therefore $\varphi(\gamma \fixed[0.2]{\text{ }}\varphi^{-1}(x)) \to 0$ as $x \to 0^{+}$ in both cases.

Hence, letting $s\to 0$ and $t\to 0$ in \eqref{eq:lambda}, we obtain, respectively,
$$
\beta_1 \fixed[0.2]{\text{ }}\Lambda_\gamma(s)=\Lambda_\gamma(\beta_1s) \ \text{ and } \ \beta_2 \fixed[0.2]{\text{ }}\Lambda_\gamma(t)=\Lambda_\gamma(\beta_2t),\text{ for all }s,t \in \mathbb{R}_{+}.
$$
Together with \eqref{eq:lambda}, this implies that
\begin{equation}\label{eq:cacuchys}
\Lambda_\gamma(x) + \Lambda_\gamma(y)=\Lambda_\gamma(x+y),\text{ for all }x,y \in \mathbb{R}_{+}.
\end{equation}
Therefore, it is known, see e.g. \cite{MR4110366} or the textbooks \cite{MR1004465, MR2467621}, 
%
%
that there exists a function $m: \mathbb{R}_{+} \to \mathbb{R}_{+}$ such that 
\begin{equation}\label{eq:biglambda}
\Lambda_\gamma(x)=m(\gamma)\fixed[0.2]{\text{ }} x,\text{ for all }x,\gamma \in \mathbb{R}_{+}.
\end{equation}



Notice also that $m$ is continuous. Indeed, it follows by \eqref{eq:defbiglambda} and \eqref{eq:biglambda} that $\varphi(\gamma  \fixed[0.2]{\text{ }}\varphi^{-1}(x))=m(\gamma) \fixed[0.2]{\text{ }}x$ for all $x,\gamma \in \mathbb{R}_{+}$. In particular, setting $x=1$, we obtain $m(\gamma)=\varphi(\gamma  \fixed[0.2]{\text{ }}\varphi^{-1}(1))$ for all $\gamma \in \mathbb{R}_{+}$. Hence, by the continuity of  $\varphi$ we get the continuity of the function $m$. 

Lastly, note that
$$
\Lambda_{\gamma_1\gamma_2}=\Lambda_{\gamma_1} \circ \Lambda_{\gamma_2},\text{ for all }\gamma_1,\gamma_2 \in \mathbb{R}_{+},
$$
with the consequence that the function $m$ is multiplicative, i.e., 
$$
m(\gamma_1\gamma_2)=m(\gamma_1) m(\gamma_2),\text{ for all }\gamma_1,\gamma_2 \in \mathbb{R}_{+}.
$$
Together with the continuity of $m$, it is straightforward (we omit details) to show the existence of a non-zero constant $c \in \mathbb{R}$ such that
$$
m(\gamma)=\gamma^c,\quad \text{ for all }\gamma \in \mathbb{R}_{+}.
$$

Therefore, it follows by \eqref{eq:defbiglambda} and \eqref{eq:biglambda} that
$$
\varphi(\gamma \fixed[0.2]{\text{ }}\varphi^{-1}(x))=\Lambda_\gamma(x)=m(\gamma)\fixed[0.2]{\text{ }}x=\gamma^c\fixed[0.2]{\text{ }}x,\text{ for all }x,\gamma \in \mathbb{R}_{+}.
$$
Note that $c\neq 0$. Setting $x=1$, we obtain
$$
\varphi(\gamma \fixed[0.2]{\text{ }}\varphi^{-1}(1))=\gamma^c, \text{ for all }\gamma \in \mathbb{R}_+,
$$
which can be rewritten as
\begin{equation}\label{eq:final}
f^{-1}\circ T=g^{-1},
\end{equation}
where $T:\mathbb{R}_{+}\to \mathbb{R}_{+}$ is the bijection defined by $T(x):= \varphi^{-1}(1)x^{1/c}$. 
Taking the inverses of each side in \eqref{eq:final}, we conclude that $f=T \circ g$, that is, $f=ag^b$ for some $a,b \in \mathbb{R}$ with $a>0$ and $b\neq 0$. Lastly, recalling that the section $\Phi(\cdot,1)$ is concave by Lemma \ref{lem:matk}, it follows that the map $\kappa: \mathbb{R}_{+}\to \mathbb{R}_{+}$ defined by $\kappa(t):=\varphi^{-1}(\beta_1\varphi(t)+\beta_2\varphi(1))$ is concave. Since $\varphi(t)=(g\circ f^{-1})(t)=(t/a)^{1/b}$ and $\varphi^{-1}(t)=at^b$ then 
$$
\kappa(t)=a\left(\beta_1 \left(\frac{t}{a}\right)^{1/b}+\beta_2\left(\frac{1}{a}\right)^{1/b}\right)^b, \text{ for all }t \in \mathbb{R}_{+}.
$$
Simple calculations show that $\kappa^{\prime\prime}(t)=\beta_1\beta_2\, \frac{1-b}{b}\left(\beta_1+\beta_2t^{-\frac{1}{b}}\right)^{b-2}t^{-1-\frac{1}{b}}$ for all $t \in \mathbb{R}_+$. 
Hence, the concavity of $\kappa$ implies $b<0$ or $b\ge 1$. Since both $f$ and $g$ are increasing, we conclude that $b\ge 1$. 

%

 \medskip
 
\textsc{If part.} 
Suppose that $f=ag^b$ with some $a,b \in \mathbb{R}^+$ with $b\ge 1$. Note that $f^{-1}(t)=g^{-1}((t/a)^{1/b})$.  
Fix also a $\mathscr{L}\otimes \mathscr{M}$-measurable simple function $h: X\times Y\to I$. 
Recalling that $g^{-1}$ is increasing and Minkowski's integral inequality, see e.g. \cite[Appendix A]{MR0290095}, cf. also \cite[Chapter 3, Section 2.4]{MR2024343}, we obtain 
$$
f^{-1}\!\left(\int_X f\!\left(g^{-1}\!\left(\int_Y g \circ h\;d\mu\right)\right)d \lambda\right)\!
=
g^{-1}\!\left(\left(\int_X \! \left(\int_Y g \circ h\;d\mu\right)^b d \lambda\right)^{1/b}\right)\\
$$
$$
\text{ }\hspace{25mm} \le 
g^{-1}\!\left(\int_Y \! \left(\int_X (g \circ h)^b\;d\lambda\right)^{1/b} d \mu\right)
=
g^{-1}\!\left(\int_Y g\!\left(f^{-1}\!\left(\int_X f \circ h\;d\lambda\right)\right)d \mu\right)\!.
$$
Therefore inequality \eqref{eq:maininequality} holds.

\bigskip

The above part solves the case where $g$ is increasing. Suppose now that $f$ is increasing and $g$ is decreasing. The whole proof goes on the same lines, with the conclusion that $f=ag^b$ for some $a \in \mathbb{R}_+$ and some $b<0$ or $b\ge 1$. 
Hence, by the assumed monotonicity of $f$ and $g$, we obtain $b<0$. 

\medskip 

Lastly, suppose that both $f$ and $g$ are decreasing. Again, the whole proof goes on the same lines noting, by Lemma \ref{lem:keylem}, that inequality \eqref{eq:fundamentalinequality} holds with the opposite direction. Hence the same function $\kappa$ above needs to be convex. Considering that $f=ag^b$ for some $a \in \mathbb{R}_+$ and some non-zero $b \in \mathbb{R}$, we conclude by the decreasingness of both $f$ and $g$ that $b \in (0,1]$. The \textsc{Only If} part of this latter case goes analogously.
\end{proof}

\begin{proof}
[Proof of Theorem \ref{th:main2}]
The proof of Theorem \ref{th:main2} goes on the same lines of the previous one. For, suppose that $g$ is increasing and repeat the same reasoning replacing $\Phi$ with $\tilde{\Phi}$, with the same conclusion that $f=ag^b$ for some $a,b \in \mathbb{R}$ with $a>0$ and $b\neq 0$. At this point, the concavity of the section 
$$
\tilde{\Phi}(t,1)=-\varphi(\alpha_1\varphi^{-1}(t)+\alpha_2\varphi^{-1}(1))=-(\alpha_1t^b+\alpha_2)^{1/b}
$$
is equivalent to $b\ge 1$. The case where $g$ is decreasing and the proofs of the \textsc{If} parts are analogous (we omit further details). 
\end{proof}

\begin{proof}
[Proof of Theorem \ref{thm:probabilitycharacterization}]
\ref{item:prob1} $\implies$ \ref{item:prob2}. By Lemma \ref{lem:keylem}\ref{item:1lem}, $f$ is increasing. Moreover, by Lemma \ref{lem:keylem}\ref{item:3lem}, the function $\Psi_{\alpha_1}$, which coincides with $-\tilde{\Phi}$, is $\beta_1$-convex by inequality \eqref{eq:fundamentalinequality3}. Since $\Psi_{\alpha_1}$ is also continuous, then it has to be convex, see e.g. \cite[Theorem 1]{MR3873086} or the textbooks \cite{MR1004465, MR2467621}.  

\medskip

\ref{item:prob2} $\implies$ \ref{item:prob3}. Let $\Gamma$ be the set of all $\gamma \in (0,1)$ for which $\Psi_\gamma$ is convex. By item \ref{item:prob2}, $\Gamma$ is non-empty, indeed $\alpha_1 \in \Gamma$. 

We claim that $\Gamma$ is relatively dense in $(0,1)$. For, 
suppose that $\gamma_1,\gamma_2 \in \Gamma$ and note that 
$$
\Psi_{\gamma_1}\left(\Psi_{\gamma_2}(\mathbf{x}),x_2\right)=\Psi_{\gamma_1\gamma_2}(\mathbf{x}), 
\text{ for all }\mathbf{x} \in D.
$$
Since the composition of increasing convex functions is convex, it follows that $\gamma_1\gamma_2 \in \Gamma$. In addition, it is clear that if $\gamma \in \Gamma$ then $1-\gamma \in \Gamma$. If $\Gamma$ were not relatively dense in $(0,1)$, there would exist a non-empty open interval $U:=(a,b)$ with maximal lenght such that $U\cap \Gamma=\emptyset$. Since $\Gamma \neq \emptyset$ then $a\neq 0$ or $b\neq 1$. Suppose that $b\neq 1$ (the other case is analogous). Let $n\ge 1$ be the smallest integer such that $\alpha_1^n<1-\nicefrac{a}{b}$. Note that, by the minimality of $n$, we have $1-\alpha_1^n \le 1-\alpha_1(1-\nicefrac{a}{b})$. Fix also a sufficiently small $\varepsilon >0$ such that $b+\varepsilon<1$ and $(b+\varepsilon)(1-\alpha_1(1-\nicefrac{a}{b}))<b$ (note that this is indeed possible). It follows by construction that  
\begin{equation}\label{eq:denseGamma}
a< b(1-\alpha_1^{n})\le b\left(1-\alpha_1(1-\nicefrac{a}{b})\right)< (b+\varepsilon)(1-\alpha_1^{n}) <b<b+\varepsilon<1.
\end{equation}
By the maximality of $U$ we can take $\gamma \in (b,b+\varepsilon)$. However, this implies that $\gamma(1-\alpha_1^n) \in \Gamma$, which contradicts \eqref{eq:denseGamma} and the fact that $U\cap \Gamma=\emptyset$.

At this point, more explicitly, we know that
$$
\Psi_\gamma\left(\alpha \mathbf{x}+(1-\alpha)\mathbf{y}\right)\le 
\alpha \Psi_\gamma\left(\mathbf{x}\right)+
(1-\alpha) \Psi_\gamma\left(\mathbf{y}\right)
$$
for all $\alpha \in (0,1)$, $\mathbf{x}, \mathbf{y} \in D$, and $\gamma \in \Gamma$. By the continuity of $f\circ g^{-1}$, we conclude that $\Gamma=(0,1)$. 

\medskip

\ref{item:prob3} $\implies$ \ref{item:prob1}.
This is straightforward.
\end{proof}


\section{Closing remark}\label{closingremarks}

In light of the proof of Theorem \ref{thm:probabilitycharacterization}, it is worth to remark that, if $(X, \mathscr{L}, \lambda)$ and $(Y, \mathscr{M}, \mu)$ are non-degenerate measure spaces with $\lambda(X)\le 1$ and $\mu(Y)\le 1$, then inequality \eqref{eq:maininequality} is satisfied by every $\mathscr{L}\otimes \mathscr{M}$-measurable simple func\-tion $h \colon X\times Y \to I$, provided that condition \ref{item:prob2} or condition \ref{item:prob3} holds. 


\bibliographystyle{amsplain}

\providecommand{\href}[2]{#2}

\end{document}